\documentclass{article}
\usepackage{amsmath, amsthm, amssymb}
\usepackage[margin=3.5cm]{geometry}
\usepackage{hyperref} %
\usepackage{verbatim}
\usepackage{url}
\usepackage{yfonts}
\usepackage{youngtab}
\usepackage{shuffle}

\usepackage{tikz, xcolor, mathrsfs, multicol}
\usetikzlibrary{shapes.geometric, arrows}

\usepackage{enumerate}

\newtheorem{thm}{Theorem}[subsection]

\newtheorem{lem}[thm]{Lemma}

\newtheorem{mydef}[thm]{Definition}
\theoremstyle{remark}
\newtheorem{rem}[thm]{Remark}
\newtheorem{notn}[thm]{Notation}

\newtheorem{exm}[thm]{Example}

\title{Parking functions on directed stars and orientation reversal with an extension to general directed trees}
\author{Roger Tian\footnote{\href{mailto:rgtian@math.ucdavis.edu}{rgtian@math.ucdavis.edu} or \href{mailto:htrland@gmail.com}{htrland@gmail.com}}}
\date{\today}

\begin{document}
\maketitle

\begin{abstract}
Parking functions, classically defined in terms of cars with preferred parking spots on a directed path attempting to park there, arise in many combinatorial situations and have seen various generalizations. In particular, parking functions have been defined for general digraphs, which yields many more enumeration problems. For example, in a directed tree whose edges are orientated away from the root, it is unknown in general how the number of parking functions on it changes once the orientation is reversed, even in the case when the tree is a star. We show that this orientation reversal results in more parking functions on the directed star in most cases, after which we extend these methods to show that this also results in more parking functions on the general directed tree if, in some sense, the number of vertices greatly exceeds the number of cars.
\end{abstract}

\section{Introduction}
Introduced by Konheim and Weiss \cite{KonWei} in their work on the linear probing solution to collisions on hash tables, parking functions can be described as follows: Consider a sequence of $m$ cars attempting to park, one after another, randomly along a one-way street with $n$ parking spots. Each car has a preferred parking spot and will park there if the spot is unoccupied. If it is occupied, the car attempts to park at the next parking spot, and this process continues until the car manages to park, or it terminates due to lack of available parking spots. The street can be modeled by a directed path with $n$ vertices, and this sequence $s \in [n]^m$ of car preferences is called a \textit{classical $(m,n)$-parking function} if all $m$ cars manage to park through this process. For $m \leq n$, there are $(n-m+1)(n+1)^{m-1}$ classical $(m,n)$-parking functions \cite{CaJoSch}. 

Classical parking functions are prominent combinatorial objects in the study of topics such as noncrossing partitions, tree enumeration, and acyclic functions \cite{GesSeo} \cite{Stan} \cite{Yan15}. There have been many generalizations of parking functions, including rational parking functions \cite{ArmLeoWar}, Naples parking functions \cite{Baum}, and parking functions on digraphs \cite{KingYan}. In particular, King and Yan \cite{KingYan} investigated what happens to the number of parking functions on a directed tree when the orientation is reversed, and we will continue this line of study here.

For a digraph $D$ let $P(D,m)$ denote the number of parking functions with $m$ cars on the digraph $D$, and let $n^{\underline{m}} := {n \choose m}m!$. In the case where the digraph is a sink tree $T$, i.e. a rooted tree with vertex set $[n]$ and edges oriented toward the root, the following inequalities have been established \cite{LackPan}: \[n^{\underline{m}}+{m \choose 2}(n-1)^{\underline{m-1}} \leq P(T,m) \leq (n-m+1)(n+1)^{m-1}.\]

Let $\tilde{T}$ be the source tree obtained by reversing the orientation of $T$. Then the following inequalities hold \cite{KingYan}: \[\sum_{i=0}^{m}{{m \choose i}(n-1)^{\underline{m-i}}} \leq P(\tilde{T},m) \leq (n-m+1)(n+1)^{m-1}\] and furthermore \[P(T,n) \leq P(\tilde{T},n)\] with equality if and only if $T$ is a path.

It is not known in general which of $P(T,m)$ and $P(\tilde{T},m)$ is larger for $m < n$, even in the case where $T$ is a star whose center is the root, for which $P(T,m) = n^{\underline{m}}+{m \choose 2}(n-1)^{\underline{m-1}}$ and $P(\tilde{T},m) = \sum_{i=0}^{m}{{m \choose i}(n-1)^{\underline{m-i}}}$.

In Section \ref{dirstar}, we show that the sink star has more parking functions than its corresponding source star in most cases. These methods will then be extended in Section \ref{dirtree} to show that $P(\tilde{T},m) < P(T,m)$ when $m$ is ``much smaller'' than $n$, for any sink tree $T$.

\section{Acknowledgments}
The author would like to thank Westin King and Catherine Yan for helpful conversations via email.

\section{Preliminaries}
We start with the definition of parking functions on digraphs given in \cite{KingYan}.

\begin{mydef}[Parking Process]
Let $m, n$ be integers such that $0 \leq m \leq n$. Let $D$ be a digraph on vertex set $[n]$ and let $s \in [n]^m$. One after another $m$ cars try to park in $D$ via the following process:
\begin{enumerate}
\item Car $i$ starts at vertex $s_i$.
\item \label{proctwo} The car will park at the current vertex if it is unoccupied. If the current vertex is occupied, then the car chooses a vertex in its neighborhood and drives there.
\item The car repeats step \ref{proctwo}) until it either parks, and the next car enters, or is unable to find an available vertex to park at, and the process terminates.
\end{enumerate}
\end{mydef}

\begin{mydef}
Let $D$ be a digraph on vertex set $[n]$. For a sequence $s \in [n]^m$, we say that $s$ is a \textbf{parking function on $D$} if it is possible for all $m$ cars to park following the aforementioned parking process. If $s$ is a parking function on $D$, then the pair $(D,s)$ will be called an \textbf{$(n,m)$-parking function}.
\end{mydef}

\begin{exm} Suppose $D$ is the following digraph.

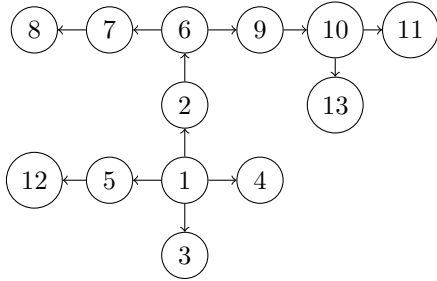
\begin{tikzpicture}
\node[shape=circle,draw=black] (1) at (0,0) {1};
\node[shape=circle,draw=black] (2) at (0,1) {2};
\node[shape=circle,draw=black] (3) at (0,-1) {3};
\node[shape=circle,draw=black] (4) at (1,0) {4};
\node[shape=circle,draw=black] (5) at (-1,0) {5};
\node[shape=circle,draw=black] (6) at (0,2) {6};
\node[shape=circle,draw=black] (7) at (-1,2) {7};
\node[shape=circle,draw=black] (8) at (-2,2) {8};
\node[shape=circle,draw=black] (9) at (1,2) {9};
\node[shape=circle,draw=black] (10) at (2,2) {10};
\node[shape=circle,draw=black] (11) at (3,2) {11};
\node[shape=circle,draw=black] (12) at (-2,0) {12};
\node[shape=circle,draw=black] (13) at (2,1) {13};

\path [->] (1) edge (2);
\path [->] (1) edge (3);
\path [->] (1) edge (4);
\path [->] (1) edge (5);
\path [->] (2) edge (6);
\path [->] (6) edge (7);
\path [->] (7) edge (8);
\path [->] (9) edge (10);
\path [->] (10) edge (11);
\path [->] (5) edge (12);
\path [->] (6) edge (9);
\path [->] (10) edge (13);
\end{tikzpicture}

For $n = 13$ and $m=9$, the sequence $s = (6,6,6,10,10,1,1,1,1)$ is a $(13,9)$-parking function.
\end{exm}

\noindent We now go over some notations that will be used repeatedly in the following sections.

\begin{notn}
\label{numset}
For vertices $i, j \in [n]$, we write $i \preceq_D j$ if there is a directed path from $i$ to $j$ in $D$. In the vein of \cite{KingYan}, we continue the practice of using $P(D,m)$ to denote the number of parking functions with $m$ cars on the digraph $D$. We will use $\breve{P}(D,m)$ to denote the set of parking functions with $m$ cars on the digraph $D$. More generally, if ``$P\ldots$'' denotes the number of parking functions satisfying some conditions, we will use ``$\breve{P}\ldots$'' to denote the set of parking functions satisfying these conditions, and vice versa.
\end{notn}

\section{Directed Stars \label{dirstar}}
Fix $3 < m < n \in \mathbb{N}$. Let $S$ denote a sink star with $n$ vertices and root $z$, and let $\tilde{S}$ denote the corresponding source star. We wish to find out which of \[P(\tilde{S},m) = \sum_{i=0}^{m}{{m \choose i}(n-1)^{\underline{m-i}}}\;,\] and \[P(S,m) = n^{\underline{m}}+{m \choose 2}(n-1)^{\underline{m-1}}\] is larger.

In comparing $P(S,m)$ and $P(\tilde{S},m)$, there are three cases for a parking function to consider, depending on whether the parking lot is $\tilde{S}$ or $S$:
\begin{enumerate}
\item No car parks at $z$. This applies to both $S$ and $\tilde{S}$.
\item One car prefers $z$. This applies to both $S$ and $\tilde{S}$.
\item  
\begin{enumerate}
\item Two cars prefer one leaf. This applies to $S$.
\item At least two cars prefer $z$. This applies to $\tilde{S}$.
\end{enumerate}
\end{enumerate}

The following lemma shows that the number of parking functions is the same in Cases 1 and 2, where no two cars prefer the same vertex.

\begin{lem}
\label{cancelcom}
We have \[\sum_{i=0}^{1}{{m \choose i}(n-1)^{\underline{m-i}}} = n^{\underline{m}}\;.\]
\end{lem}
\begin{proof}
Each parking function $(\tilde{S},s)$ on $\tilde{S}$ with no two cars preferring the same vertex corresponds to a unique parking function $(S,s)$ on $S$ with no two cars preferring the same vertex, and vice versa.
\end{proof}

Let $\breve{P}_{\geq2}(\tilde{S},m)$ denote the set of elements of $\breve{P}(\tilde{S},m)$ with at least two cars preferring $z$. Let $\breve{P}_{0\&1}(S,m)$ denote the set of elements of $\breve{P}(S,m)$ with two cars preferring the same leaf. To compare \[\sum_{i=0}^{m}{{m \choose i}(n-1)^{\underline{m-i}}}\] and \[n^{\underline{m}}+{m \choose 2}(n-1)^{\underline{m-1}}\;,\] it suffices to compare \[P_{\geq2}(\tilde{S},m) = \sum_{i=2}^{m}{{m \choose i}(n-1)^{\underline{m-i}}}\] and \[P_{0\&1}(S,m) = {m \choose 2}(n-1)^{\underline{m-1}}\] by Lemma \ref{cancelcom}; in other words, any difference must come from Case 3.

Due to the apparent difficulty in comparing $P_{\geq2}(\tilde{S},m)$ and $P_{0\&1}(S,m)$ directly from their formulae, we will partition $\breve{P}_{\geq2}(\tilde{S},m)$ and $\breve{P}_{0\&1}(S,m)$ into corresponding subsets that are much easier to compare; this partitioning in effect controls the ``common'' aspects of $P_{\geq2}(\tilde{S},m)$ and $P_{0\&1}(S,m)$, and allows us to focus on where they are different. Let $i < j \in [m]$. We first define the sets $\breve{P}_{\geq2}(\tilde{S},m)^{i,j}$ partitioning $\breve{P}_{\geq2}(\tilde{S},m)$ and the sets $\breve{P}_{0\&1}(S,m)^{i,j}$ partitioning $\breve{P}_{0\&1}(S,m)$ as follows.

Let $\breve{P}_{0\&1}(S,m)^{i,j}$ denote the set of parking functions in $\breve{P}_{0\&1}(S,m)$ with cars $i$, $j$ preferring the same leaf vertex. Let $U_{i,j} := [j]-\{i,j\}$. Let $\breve{P}_{\geq2}(\tilde{S},m)^{i,j}$ denote the set of parking functions in $\breve{P}_{\geq2}(\tilde{S},m)$ with $(i,j)$ being the smallest pair of cars preferring the root, i.e. there are no cars in $U_{i,j}$ preferring the root. We further partition $\breve{P}_{0\&1}(S,m)^{i,j}$ and $\breve{P}_{\geq2}(\tilde{S},m)^{i,j}$ as follows.

For any parking function in $\breve{P}_{\geq2}(\tilde{S},m)^{i,j}$, the cars in $U_{i,j}$ prefer leaf vertices, so we can first arrange these $j-2$ cars among the $n-1$ leaves. For each injection $f: U_{i,j} \rightarrow \tilde{S}-\{z\}$, let $\breve{P}_{\geq2}(\tilde{S},m)^{i,j}(f)$ denote the elements of $\breve{P}_{\geq2}(\tilde{S},m)^{i,j}$ with cars in $U_{i,j}$ arranged via $f$, and let $\breve{P}_{0\&1}(S,m)^{i,j}(f)$ denote the elements of $\breve{P}_{0\&1}(S,m)^{i,j}$ with cars in $U_{i,j}$ arranged via $f$. Since $f$ is itself a ``restricted'' parking function, we will express $f$ as a sequence when the context is clear. We can thus partition $\breve{P}_{\geq2}(\tilde{S},m)^{i,j}$ into the sets $\breve{P}_{\geq2}(\tilde{S},m)^{i,j}(f)$ and $\breve{P}_{0\&1}(S,m)^{i,j}$ into the sets $\breve{P}_{0\&1}(S,m)^{i,j}(f)$, with $f$ ranging over the injections $U_{i,j} \rightarrow \tilde{S}-\{z\}$.

\begin{exm} Suppose $\tilde{S}$ is the following source star.

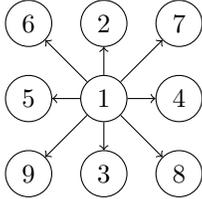
\begin{tikzpicture}
\node[shape=circle,draw=black] (1) at (0,0) {1};
\node[shape=circle,draw=black] (2) at (0,1) {2};
\node[shape=circle,draw=black] (3) at (0,-1) {3};
\node[shape=circle,draw=black] (4) at (1,0) {4};
\node[shape=circle,draw=black] (5) at (-1,0) {5};
\node[shape=circle,draw=black] (6) at (-1,1) {6};
\node[shape=circle,draw=black] (7) at (1,1) {7};
\node[shape=circle,draw=black] (8) at (1,-1) {8};
\node[shape=circle,draw=black] (9) at (-1,-1) {9};

\path [->] (1) edge (2);
\path [->] (1) edge (3);
\path [->] (1) edge (4);
\path [->] (1) edge (5);
\path [->] (1) edge (6);
\path [->] (1) edge (7);
\path [->] (1) edge (8);
\path [->] (1) edge (9);
\end{tikzpicture}

Let $m = 5$, $(i,j) = (1,3)$, and $f = (6)$. On $\tilde{S}$, the cars $(i,j)$ must prefer vertex $1$. On $S$, the cars $(i,j)$ can prefer vertices $2, 7, 4, 8, 3, 9, 5$. An element of $\breve{P}_{\geq2}(\tilde{S},m)^{i,j}(f)$ is $s = (1,6,1,7,1)$. An element of $\breve{P}_{0\&1}(S,m)^{i,j}(f)$ is $s = (8,6,8,4,3)$.
\end{exm}

The following lemma will help us compare the cardinalities of the sets $\breve{P}_{\geq2}(\tilde{S},m)^{i,j}(f)$ and $\breve{P}_{0\&1}(S,m)^{i,j}(f)$.
\begin{lem}
\label{precise}
For integers $a, b$ with 
$0 \leq a \leq b-1$ and $a+1 \leq \frac{b+1}{r}$, where $r \geq 2$ is a positive integer, we have \[r\sum_{l=0}^{a}{{a \choose l}(b+1)^{\underline{l}}} \leq (b+1)b^{\underline{a}},\] with equality possible only when $a = 0$.
\end{lem}
\begin{proof}
The inequality is clear for $a = 0$. For $a = 1$, we have $r\sum_{l=0}^{1}{{1 \choose l}(b+1)^{\underline{l}}} = r(b+2)$ and $(b+1)b^{\underline{1}} = b(b+1)$, so the inequality holds as $b(b+1) \geq 2rb > r(b+2)$ since $b \geq 3$. For $a = 2$, we have $r\sum_{l=0}^{2}{{2 \choose l}(b+1)^{\underline{l}}} = r[1+2(b+1)+b(b+1)] = r(b^2+3b+3)$ and $(b+1)b^{\underline{2}} = (b+1)b(b-1) \geq 3rb(b-1)$, so the inequality holds as $3b(b-1) > b^2+3b+3$ since $b \geq 5$.

Now assume $a > 2$. We show that $$\sum_{l=0}^{a}{{a \choose l}(b+1)^{\underline{l}}} < (a+1)b^{\underline{a}},$$ by comparing $C_l = {a \choose l}(b+1)^{\underline{l}}$ with $b^{\underline{a}}$ for $l = 0, 1, \ldots a$. 

We have $C_{a-2} = \frac{a(a-1)}{2}(b+1)^{\underline{a-2}} = b^{\underline{a-3}}[a(a-1)\frac{b+1}{2}]$ and $b^{\underline{a}} = b^{\underline{a-3}}[(b-a+3)(b-a+2)(b-a+1)]$. Since $a \leq \frac{b+1}{r}-1 \leq \frac{b}{2}$, we have $[a(a-1)\frac{b+1}{2}] < [(b-a+3)(b-a+2)(b-a+1)]$, and so $C_{a-2} < b^{\underline{a}}$.

It is straightforward to check that $C_l < C_{a-2} < b^{\underline{a}}$ for all $l < a-2$. Hence it remains to compare $C_0+C_1+C_{a-1}+C_{a}$ with $b^{\underline{a}}+b^{\underline{a}}+b^{\underline{a}}+b^{\underline{a}} = 4b^{\underline{a}}$.

We have $4b^{\underline{a}} - (C_0+C_1+C_{a-1}+C_{a}) = 4b^{\underline{a}}-(b+1)^{\underline{a}}-a(b+1)^{\underline{a-1}}-a(b+1)-1 = b^{\underline{a-2}}[4(b-a+1)(b-a+1)-(b+1)(b-a+2)-a(b+1)]-a(b+1)-1 = b^{\underline{a-2}}[(b-a+2)(3b-4a+3)-a(b+1)]-a(b+1)-1$. Since $a \leq \frac{b}{2}$, we have $(b-a+2)(3b-4a+3)-a(b+1) \geq (b/2+2)(b+3)-(b/2)(b+1) = 3b+6$. It follows that $4b^{\underline{a}} - (C_0+C_1+C_{a-1}+C_{a}) \geq b^{\underline{a-2}}(3b+6)-a(b+1)-1 > 0$.

This proves that $\sum_{l=0}^{a}{{a \choose l}(b+1)^{\underline{l}}} < (a+1)b^{\underline{a}}$. It follows that $r\sum_{l=0}^{a}{{a \choose l}(b+1)^{\underline{l}}} < r(a+1)b^{\underline{a}} \leq (b+1)b^{\underline{a}}$ as desired.
\end{proof}

\begin{thm}
\label{precisefinal}
Let $3 < m \leq \frac{n+1}{r}$ where $r \geq 2$ is a positive integer. Then \[r\sum_{i=2}^{m}{{m \choose i}(n-1)^{\underline{m-i}}} < {m \choose 2}(n-1)^{\underline{m-1}}.\]
\end{thm}
\begin{proof}
As described previously, partition $\breve{P}_{\geq2}(\tilde{S},m)$ into the sets $\breve{P}_{\geq2}(\tilde{S},m)^{i,j}(f)$ and partition $\breve{P}_{0\&1}(S,m)$ into the sets $\breve{P}_{0\&1}(S,m)^{i,j}(f)$. To prove the inequality, we compare $P_{\geq2}(\tilde{S},m)^{i,j}(f)$ and $P_{0\&1}(S,m)^{i,j}(f)$ for a fixed injection $f: U_{i,j} \rightarrow \tilde{S}-\{z\}$.

Fix a leaf vertex $x$. Note that an element of $\breve{P}_{\geq2}(\tilde{S},m)^{i,j}(f)$ is completely determined once we finish arranging all the cars preferring the leaves, and that an element of $\breve{P}_{0\&1}(S,m)^{i,j}(f)$ is completely determined once we finish arranging the car pair $(i,j)$ and the remaining cars. For each integer $0 \leq l \leq m-j$, let $\breve{P}_{\geq2}(\tilde{S},m)^{i,j}_l(f)$ denote the elements of $\breve{P}_{\geq2}(\tilde{S},m)^{i,j}(f)$ with exactly $|U_{i,j}|+l = j+l-2$ cars preferring leaves; i.e. we choose $l$ cars from $[m]-[j]$ to arrange them among the $n-1-|U_{i,j}| = n-j+1$ remaining leaves. If $j > 2$, let $\breve{P}_{0\&1}(S,m)^{i,j}_k(f)$ denote the elements of $\breve{P}_{0\&1}(S,m)^{i,j}(f)$ with car pair $(i,j)$ preferring the $(k+1)$st empty leaf after car $\min U_{i,j}$ clockwise. If $j = 2$, let $\breve{P}_{0\&1}(S,m)^{i,j}_k(f)$ denote the elements of $\breve{P}_{0\&1}(S,m)^{i,j}(f)$ with car pair $(i,j)$ preferring the $(k+1)$st leaf after $x$ clockwise. We can thus partition $\breve{P}_{\geq2}(\tilde{S},m)^{i,j}(f)$ into the sets $\breve{P}_{\geq2}(\tilde{S},m)^{i,j}_l(f)$, for $0 \leq l \leq m-j$. On the other hand, the sets $\breve{P}_{0\&1}(S,m)^{i,j}_k(f)$ partition $\breve{P}_{0\&1}(S,m)^{i,j}(f)$, for $0 \leq k \leq n-j$.

Notice that \[P_{\geq2}(\tilde{S},m)^{i,j}(f) = \sum_{l=0}^{m-j}{P_{\geq2}(\tilde{S},m)^{i,j}_l(f)} = \sum_{l=0}^{m-j}{{m-j \choose l}(n-j+1)^{\underline{l}}}\] and \[P_{0\&1}(S,m)^{i,j}(f) = \sum_{k=0}^{n-j}{P_{0\&1}(S,m)^{i,j}_k(f)} = (n-j+1)(n-j)^{\underline{m-j}}.\] Since $r \geq 2$ and $j \geq 2$, it is easy to check that $m-j+1 \leq \frac{n+1}{r}-j+1 = \frac{n+1-rj+r}{r} \leq \frac{n-j+1}{r}$. Letting $a = m-j$ and $b = n-j$, we have $a \leq b-1$ and $b \leq n-2$, and it follows by Lemma \ref{precise} that $rP_{\geq2}(\tilde{S},m)^{i,j}(f) \leq P_{0\&1}(S,m)^{i,j}(f)$ for all $i < j \in [m]$ and all injections $f: U_{i,j} \rightarrow [n-1]$, with equality possible only in the case $j=m$. Consequently, we have \[\sum_{i<j}{\sum_{f}{rP_{\geq2}(\tilde{S},m)^{i,j}(f)}} < \sum_{i<j}{\sum_{f}{P_{0\&1}(S,m)^{i,j}(f)}}\] and hence \[r\sum_{i=2}^{m}{{m \choose i}(n-1)^{\underline{m-i}}} < {m \choose 2}(n-1)^{\underline{m-1}}\] as desired.
\end{proof}

The following lemma will help us obtain a different inequality dealing with larger values of $m$.
\begin{lem}
\label{premaxbound}
For integers $a$ and $b$ with $0 \leq a \leq b-1$ and $a \leq \frac{2}{3}b$, we have \[\sum_{l=0}^{a}{{a \choose l}(b+1)^{\underline{l}}} < (b+1)b^{\underline{a}}\]
\end{lem}
\begin{proof}
Notice that $(b+1)b^{\underline{a}} = (b+1)^{\underline{a+1}} = (b+1)^{\underline{a}}(b-a+1)$. We have the following crude upper bounds for $C_l = {a \choose l}(b+1)^{\underline{l}}$:
\begin{enumerate}
\item $C_a = (b+1)^{\underline{a}}$
\item $C_{a-1} = \frac{a}{b-a+2}(b+1)^{\underline{a}} \leq \frac{\frac{2}{3}b}{\frac{1}{3}b+2}(b+1)^{\underline{a}} \leq 2(b+1)^{\underline{a}}$
\item $C_{a-2} = \frac{a(a-1)}{2(b-a+3)(b-a+2)}(b+1)^{\underline{a}} \leq \frac{\frac{2}{3}b(\frac{2}{3}b-1)}{2(\frac{1}{3}b+3)(\frac{1}{3}b+2)} \leq 2(b+1)^{\underline{a}}$
\item $C_{a-3} = \frac{a(a-1)(a-2)}{3 \cdot 2(b-a+4)(b-a+3)(b-a+2)}(b+1)^{\underline{a}} \leq \frac{(\frac{2}{3}b)(\frac{2}{3}b-1)(\frac{2}{3}b-2)}{3\cdot2(\frac{1}{3}b+4)(\frac{1}{3}b+3)(\frac{1}{3}b+2)}(b+1)^{\underline{a}} \leq \frac{4}{3}(b+1)^{\underline{a}}$
\item Continuing in this manner, it is straightforward to check that \[C_{a-i} \leq (\frac{2}{3})(\frac{4}{15})^{i-4}(b+1)^{\underline{a}}\] for $i \geq 4$.
\end{enumerate}
It follows that $\sum_{l=0}^{a}{{a \choose l}(b+1)^{\underline{l}}} \leq (5 + \frac{4}{3}+ \frac{2}{3}\sum_{i=0}^{\infty}{(\frac{4}{15})^i})(b+1)^{\underline{a}} = (2\frac{8}{33}+5)(b+1)^{\underline{a}} = (7\frac{8}{33})(b+1)^{\underline{a}}$, and hence $\sum_{l=0}^{a}{{a \choose l}(b+1)^{\underline{l}}} < (b+1)^{\underline{a}}(b-a+1)$ whenever $b-a \geq 7$.

In the case $b-a \leq 6$, we have $a \leq \frac{2}{3}b \leq 2(b-a) \leq 12$. For $b-a = 6, 5, \ldots, 1$ we can choose $a = 2(b-a)$ and it is straightforward to check individually that the bound $\sum_{l=0}^{a}{{a \choose l}(b+1)^{\underline{l}}} = [1+\frac{a}{b-a+2}+\frac{a(a-1)}{2(b-a+3)(b-a+2)}+\frac{a(a-1)(a-2)}{3 \cdot 2(b-a+4)(b-a+3)(b-a+2)}+\ldots](b+1)^{\underline{a}} < (b-a+1)(b+1)^{\underline{a}}$ holds.
\end{proof}

\begin{thm}
\label{maxbound}
If $3 < m \leq \frac{2}{3}n$, then \[\sum_{i=2}^{m}{{m \choose i}(n-1)^{\underline{m-i}}} < {m \choose 2}(n-1)^{\underline{m-1}}\] or equivalently \[\sum_{i=0}^{m}{{m \choose i}(n-1)^{\underline{m-i}}} < n^{\underline{m}}+{m \choose 2}(n-1)^{\underline{m-1}}.\]
\end{thm}
\begin{proof}
Partition $\breve{P}_{\geq2}(\tilde{S},m)$ into the sets $\breve{P}_{\geq2}(\tilde{S},m)^{i,j}(f)$ and partition $\breve{P}_{0\&1}(S,m)$ into the sets $\breve{P}_{0\&1}(S,m)^{i,j}(f)$. We have $P_{\geq2}(\tilde{S},m)^{i,j}(f) = \sum_{l=0}^{m-j}{{m-j \choose l}(n-j+1)^{\underline{l}}}$ and $P_{0\&1}(S,m)^{i,j}(f) = (n-j+1)(n-j)^{\underline{m-j}}.$ Notice that $m-j \leq \frac{2}{3}n-j = \frac{2n-3j}{3} \leq \frac{2}{3}(n-j)$. Letting $a = m-j$ and $b = n-j$, it follows by Lemma \ref{premaxbound} that $P_{\geq2}(\tilde{S},m)^{i,j}(f) < P_{0\&1}(S,m)^{i,j}(f)$ for all $i < j \in [m]$ and all injections $f: U_{i,j} \rightarrow [n-1]$. Consequently, we have \[\sum_{i<j}{\sum_{f}{P_{\geq2}(\tilde{S},m)^{i,j}(f)}} < \sum_{i<j}{\sum_{f}{P_{0\&1}(S,m)^{i,j}(f)}}\] and hence \[\sum_{i=2}^{m}{{m \choose i}(n-1)^{\underline{m-i}}} < {m \choose 2}(n-1)^{\underline{m-1}}\] as desired.
\end{proof}


\section{General Directed Trees \label{dirtree}}
We now consider the situation where a general directed tree is sparsely-parked, in other words there are far fewer cars than parking spots in some sense. This situation has remarkable similarities with the directed star case analyzed previously, and can be studied via an extension of the previous approach. 

Let $T$ denote a sink tree with $n$ vertices with root $z$, and let $\tilde{T}$ denote the corresponding source tree. For two vertices $u, v$ of $\tilde{T}$ where $v$ is reachable from $u$, let $\mathrm{path}(u,v)$ denote the unique directed path connecting them; $\mathrm{path}(u,v)$ is defined analogously for $T$. Let $\mathrm{leaf}(u)$ denote the minimal leaf connected to $u$ by a directed path of maximal length. Let $  \mathrm{minleafdist}(T)  $ denote the number of vertices in the shortest path from a leaf to a vertex of degree at least three in $T$; $  \mathrm{minleafdist}(\tilde{T})  $ is defined the same way.

If $I$ is a directed path in $\tilde{T}$ starting at $u$, define the neighborhood $N(I)$ of $I$ to be the set of vertices adjacent to $I$, excluding the parent of $u$; intuitively, an element of $N(I)$ is the start of a branch attached to $I$.

\begin{exm} Suppose $\tilde{T}$ is the following digraph.

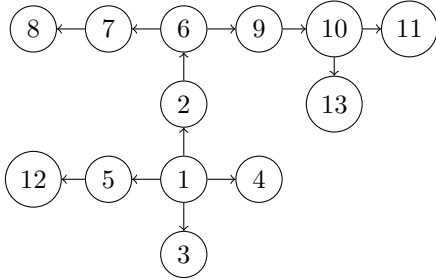
\begin{tikzpicture}
\node[shape=circle,draw=black] (1) at (0,0) {1};
\node[shape=circle,draw=black] (2) at (0,1) {2};
\node[shape=circle,draw=black] (3) at (0,-1) {3};
\node[shape=circle,draw=black] (4) at (1,0) {4};
\node[shape=circle,draw=black] (5) at (-1,0) {5};
\node[shape=circle,draw=black] (6) at (0,2) {6};
\node[shape=circle,draw=black] (7) at (-1,2) {7};
\node[shape=circle,draw=black] (8) at (-2,2) {8};
\node[shape=circle,draw=black] (9) at (1,2) {9};
\node[shape=circle,draw=black] (10) at (2,2) {10};
\node[shape=circle,draw=black] (11) at (3,2) {11};
\node[shape=circle,draw=black] (12) at (-2,0) {12};
\node[shape=circle,draw=black] (13) at (2,1) {13};

\path [->] (1) edge (2);
\path [->] (1) edge (3);
\path [->] (1) edge (4);
\path [->] (1) edge (5);
\path [->] (2) edge (6);
\path [->] (6) edge (7);
\path [->] (7) edge (8);
\path [->] (9) edge (10);
\path [->] (10) edge (11);
\path [->] (5) edge (12);
\path [->] (6) edge (9);
\path [->] (10) edge (13);
\end{tikzpicture}

We have $z = 1$, $\mathrm{path}(2,10) = (2,6,9,10)$, $\mathrm{leaf}(2) = 11$, $N(2,6) = \{7,9\}$, $N(9,10,11) = \{13\}$, and $  \mathrm{minleafdist}(\tilde{T})   = 2$.
\end{exm}

We now introduce a certain ``flip'' operation for parking functions on a path, which will be useful for studying the effects of reversing the orientation of $\tilde{T}$. Given a parking function $(\tilde{T},s)$ and a directed path $I$ in $\tilde{T}$, define $\mathrm{flip}_I(s)$ as follows. Label the vertices of $I$ as $w_1 \preceq_I w_2 \preceq_I \ldots \preceq_I w_{n_I}$. For the cars $c_1 < c_2 < \ldots < c_{m_I}$ preferring vertices in $I$, label their preferences as $w_{f(1)}, w_{f(2)}, \ldots, w_{f(m_I)} \in I$ respectively. Then $\mathrm{flip}_I(s)$ is the sequence (not necessarily a parking function) obtained from $s$ by changing the preferences of the cars $c_1 < c_2 < \ldots < c_{m_I}$ to $w_{n_I-f(1)+1}, w_{n_I-f(2)+1}, \ldots, w_{n_I-f(m_I)+1}$ respectively and fixing the preferences of all other cars. We also define $\mathrm{flip}_I(w_k) := w_{n_I-k+1}$. Intuitively, $\mathrm{flip}_I(s)$ is obtained by reversing the preferences of the cars along $I$. 

We now define a $\mathrm{flip}^*$ operation on the entire tree that is independent of the specific parking function; this operation will yield an involution from $P(\tilde{T},m)$ to $P(T,m)$ in certain cases. Given $s \in P(\tilde{T},m)$, define $\mathrm{flip}^*(s)$ as follows:
\begin{enumerate}
\item For each $u \in N(z)$, apply $\mathrm{flip}_{\mathrm{path}(u,\mathrm{leaf(u)})}$.
\item For each $u_1 \in N(\mathrm{path}(u,\mathrm{leaf}(u)))$, apply $\mathrm{flip}_{\mathrm{path}(u_1,\mathrm{leaf}(u_1))}$.
\item For each $u_2 \in N(\mathrm{path}(u_1,\mathrm{leaf}(u_1)))$, apply $\mathrm{flip}_{\mathrm{path}(u_2,\mathrm{leaf}(u_2))}$.
\item In general, for each $u_{i+1} \in N(\mathrm{path}(u_i,\mathrm{leaf}(u_i)))$, apply $\mathrm{flip}_{\mathrm{path}(u_{i+1},\mathrm{leaf}(u_{i+1}))}$.
\end{enumerate}
For a vertex $v$ we also define $\mathrm{flip}^*(v)$ to be $\mathrm{flip}_I(v)$ where $I$ is the predetermined path $v$ is on in the above process; we will call such $I$ a \textbf{flip-path} of $\tilde{T}$. The same operation can be defined from $P(T,m)$ to $P(\tilde{T},m)$. Roughly speaking, $\mathrm{flip}^*$ fixes the cars preferring $z$ in both situations, while reflecting the preferences on the ``branches''. We will be interested in situations where applying $\mathrm{flip}^*$ to $s$ and reversing orientation produces a parking function on $T$.

\begin{rem}
A type of ``flip'' operation had been used in \cite{KingYan} to show that $P(T,n) \leq P(\tilde{T},n)$. In that situation, the flip of preferences depends on the specific parking function.
\end{rem}

\begin{exm}
 Suppose $\tilde{T}$ is the following digraph.

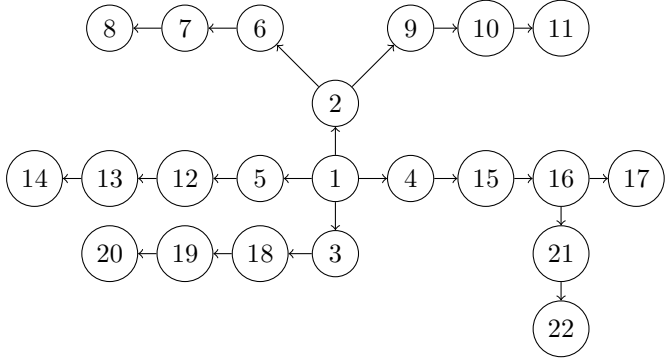
\begin{tikzpicture}
\node[shape=circle,draw=black] (1) at (0,0) {1};
\node[shape=circle,draw=black] (2) at (0,1) {2};
\node[shape=circle,draw=black] (3) at (0,-1) {3};
\node[shape=circle,draw=black] (4) at (1,0) {4};
\node[shape=circle,draw=black] (5) at (-1,0) {5};
\node[shape=circle,draw=black] (6) at (-1,2) {6};
\node[shape=circle,draw=black] (7) at (-2,2) {7};
\node[shape=circle,draw=black] (8) at (-3,2) {8};
\node[shape=circle,draw=black] (9) at (1,2) {9};
\node[shape=circle,draw=black] (10) at (2,2) {10};
\node[shape=circle,draw=black] (11) at (3,2) {11};
\node[shape=circle,draw=black] (12) at (-2,0) {12};
\node[shape=circle,draw=black] (13) at (-3,0) {13};
\node[shape=circle,draw=black] (14) at (-4,0) {14};
\node[shape=circle,draw=black] (15) at (2,0) {15};
\node[shape=circle,draw=black] (16) at (3,0) {16};
\node[shape=circle,draw=black] (17) at (4,0) {17};
\node[shape=circle,draw=black] (18) at (-1,-1) {18};
\node[shape=circle,draw=black] (19) at (-2,-1) {19};
\node[shape=circle,draw=black] (20) at (-3,-1) {20};
\node[shape=circle,draw=black] (21) at (3,-1) {21};
\node[shape=circle,draw=black] (22) at (3,-2) {22};

\path [->] (1) edge (2);
\path [->] (1) edge (3);
\path [->] (1) edge (4);
\path [->] (1) edge (5);
\path [->] (2) edge (6);
\path [->] (6) edge (7);
\path [->] (7) edge (8);
\path [->] (2) edge (9);
\path [->] (9) edge (10);
\path [->] (10) edge (11);
\path [->] (5) edge (12);
\path [->] (12) edge (13);
\path [->] (13) edge (14);
\path [->] (4) edge (15);
\path [->] (15) edge (16);
\path [->] (16) edge (17);
\path [->] (3) edge (18);
\path [->] (18) edge (19);
\path [->] (19) edge (20);
\path [->] (16) edge (21);
\path [->] (21) edge (22);
\end{tikzpicture}

For the parking function $s = (2,2,10,11)$, we have $\mathrm{flip}^*(s) = (8,8,10,9)$. For the parking function $s = (6,6,6,14,12)$, we have $\mathrm{flip}^*(s) = (7,7,7,5,13)$. For the parking function $s = (4,16,16,22)$, we have $\mathrm{flip}^*(s) = (17,15,15,21)$.
\end{exm}

In our comparison of $P(\tilde{T},m)$ and $P(T,m)$, complicated sets of parking functions, for which no nice formulae for cardinalities are known, will need to be considered. Nonetheless, the following crude estimates will be enough for our purposes.

We will be considering the situation where $j \leq m$ cars have already made their preferences in $T$ and $\tilde{T}$. The preferences of these $j$ cars yield a ``restricted'' parking function $g[T,j]: [j] \rightarrow T$ and $g[\tilde{T},j]: [j] \rightarrow \tilde{T}$ on $T$ and $\tilde{T}$ respectively. We are interested in analyzing how many ways the remaining cars $[m]-[j]$ can make their preferences to ``complete'' $g[T,j]$ to an element of $\breve{P}(T,m)$ and $g[\tilde{T},j]$ to an element of $\breve{P}(\tilde{T},m)$. 

\begin{lem}
\label{crudebounds}
Let $G$ be a digraph with $n$ vertices and $m$ cars trying to park on it. Let $\breve{P}(G,m)[g]$ denote the subset of $\breve{P}(G,m)$ whose cars $[j]$ have already made their preferences according to a predetermined parking function $g: [j] \rightarrow G$. 
\begin{enumerate}
\item \label{pfestimates} For any $l_1 < l_2 \leq m-j$ we have \[P(G,j+l_1)[g] \leq \frac{P(G,j+l_2)[g]}{(n-j-l_1)^{\overline{l_2-l_1}}}.\] 
\item \label{mustpark} Suppose $G$ is acyclic. Let $l \leq m-j$ and let $v \in [n]$ be a vertex that no cars in $[j]$ will occupy but whose outgoing paths will be completely occupied by cars in $[j]$. Suppose also that, for any vertex $v'$ with $\deg(v') \geq 3$ connected to $v$, $\mathrm{path(v',v)}$ has at least $m-j$ vertices. Let $\breve{P}(G,j+l)[g,\hat{v}]$ be the subset of $\breve{P}(G,j+l)[g]$ such that no car will park at $v$, and let $\breve{P}(G,j+l)[g,\bar{v}] = \breve{P}(G,j+l)[g]-\breve{P}(G,j+l)[g,\hat{v}]$ be the subset of $\breve{P}(G,j+l)[g]$ such that one car may park at $v$. Then \[P(G,j+l)[g,\bar{v}] \leq \frac{lm}{n-j-l+1}P(G,j+l)[g,\hat{v}].\]
\end{enumerate}
\end{lem}
\begin{proof}
The cars $[j]$ will occupy exactly $j$ vertices in the parking process. 
\begin{enumerate}
\item Any $s \in \breve{P}(G,j+l_1)[g]$ results in a parking arrangement in which exactly $j+l_1$ vertices are occupied. To complete $s$ to an $s' \in \breve{P}(G,j+l_2)[g]$, there are at least $n-j-l_1$ vertices that the $(j+l_1+1)$st car can prefer, at least $n-j-l_1-1$ vertices that the $(j+l_1+2)$st car can prefer, and so on. Thus, we have $(n-j-l_1)^{\overline{l_2-l_1}}P(G,j+l_1)[g] \leq P(G,j+l_2)[g]$.
\item 

Since $G$ has no cycles, there is at most one path connecting $v$ to any other vertex. For $v$ to be occupied, there must be a car that prefers an ingoing path to $v$ and ultimately parks at $v$. Let $s \in \breve{P}(G,j+l)[g,\bar{v}]$. $s$ has at most $l$ cars that can be chosen to park at $v$. Let $i_c \in [m]-[j]$ denote this chosen car and let $v_c$ denote its preferred vertex. The other $l-1$ car preferences constitute (up to relabeling) an element of $\breve{P}(G,j+l-1)[g,\hat{v}]$, because all outgoing paths from $v$ are unavailable for parking, all vertices on $\mathrm{path}(v_c,v)$ except possibly the endpoints have degree two, and there are at least $m-j-2$ vertices between $v$ and any (other) vertex of degree at least three; in other words, none of the other $l-1$ cars can possibly park at $v$. Since $i_c$ must prefer an ingoing path completely parked (except for $v$ itself) by cars having labels smaller than $i_c$ in order to park at $v$, there are at most $m$ choices for $v_c$. Hence there are at most $l \leq m-j$ ways to choose $i_c$, at most $P(G,j+l-1)[g,\hat{v}]$ ways to determine the other car preferences, and at most $m$ vertices for $i_c$ to prefer. It follows that $P(G,j+l)[g,\bar{v}] \leq lmP(G,j+l-1)[g,\hat{v}] \leq \frac{lm}{n-j-l+1}P(G,j+l)[g,\hat{v}]$, the last inequality following from Item \ref{pfestimates}.
\end{enumerate}
\end{proof}

\begin{thm}
\label{sparsepark}
Let $T$ be a sink tree with $n$ vertices having root $z$, and let $\tilde{T}$ denote the corresponding source tree. Let $m \geq 2$ be the number of cars attempting to park in the tree. If $m \leq \min(|N(z)|,  \mathrm{minleafdist}(T)  )$, then $P(T,m) > P(\tilde{T},m)$.
\end{thm}
\begin{proof}
In analogy to the directed star situation studied earlier, there are three cases for a parking function, depending on $T$ or $\tilde{T}$ being the parking lot in question:
\begin{enumerate}
\item Exactly one car prefers $z$. This applies to both $T$ and $\tilde{T}$.
\item No cars prefer $z$, no cars park at $z$. This applies to both $T$ and $\tilde{T}$.
\item 
\begin{enumerate}
\item No cars prefers $z$, one car parks at $z$. This applies to $T$.
\item At least two cars prefer $z$. This applies to $\tilde{T}$.
\end{enumerate}
\end{enumerate}
We first show that $P(\tilde{T},m) = P(T,m)$ in Cases 1 and 2, after which it will suffice to compare $P(\tilde{T},m)$ and $P(T,m)$ in Case 3. The Cases 1 and 2 correspond to the cancellations performed in the situation where each vertex is preferred by at most one car in the directed star case.

\noindent \underline{Case 1}: \\
Let $s' \in \breve{P}(T,m)$ be such a parking function with car $j_z$ preferring $z$. By definition, each flip-path $J$ has enough vertices to accommodate all the cars preferring $J$. It follows that $\mathrm{flip}^*(s') \in \breve{P}(\tilde{T},m)$ and that $\mathrm{flip}^*(s')$ is a parking function with just car $j_z$ preferring $z$.

Let $s \in \breve{P}(\tilde{T},m)$ be such a parking function with car $i_z$ preferring $z$. Since $m < \mathrm{minleafdist}(T)$, each flip-path $I$ has enough vertices to accommodate all the cars preferring $I$; these cars can all park without ever leaving $I$. It follows that $\mathrm{flip}_I(s)$ also satisfies this property, for each flip-path $I$. Hence $\mathrm{flip}^*(s) \in \breve{P}(T,m)$ and $\mathrm{flip}^*(s)$ is a parking function with just car $i_z$ preferring $z$.

$\mathrm{flip}^*$ is clearly an involution, so the two sets of parking functions have the same cardinality as claimed.

\noindent \underline{Case 2}: \\
The same argument as in Case 1 applies, except now with $s$, $s'$, $\mathrm{flip}^*(s)$, $\mathrm{flip}^*(s')$ all being parking functions with no cars preferring $z$ or parking at $z$.

\noindent \underline{Case 3}: \\
Let $\breve{P}_{\geq2}(\tilde{T},m)$ denote the set of elements of $\breve{P}(\tilde{T},m)$ with at least two cars preferring $z$, corresponding to Case 3b. Let $\breve{P}_{0\&1}(T,m)$ denote the set of elements of $\breve{P}(T,m)$ with no cars preferring $z$ and one car parking at $z$, corresponding to Case 3a. Similar to our work on the directed star case, we will partition $\breve{P}_{\geq2}(\tilde{T},m)$ into subsets and match them with corresponding subsets of $\breve{P}_{0\&1}(T,m)$ in such a way that makes for easier comparison; as before, this method involves controlling the ``common'' features of $\breve{P}_{\geq2}(\tilde{T},m)$ and $\breve{P}_{0\&1}(T,m)$. 

Let $i < j \in [m]$, and let $U_{i,j} := [j]-\{i,j\}$, with the following purpose. In Case 3b, the pair $i < j$ will be the first two cars preferring $z$, so all cars in $U_{i,j}$ prefer the graph $\tilde{T}-\{z\}$. In Case 3a, the pair $i < j$ will be the first two cars preferring a single vertex such that $j$ parks at $z$, so all cars in $[j]$ prefer the graph $T-\{z\}$. In both cases, the cars in $U_{i,j}$ prefer non-root vertices of the tree. 

Define $\breve{P}_{\geq2}(\tilde{T},m)^{i,j}$ to be the set of elements of $\breve{P}_{\geq2}(\tilde{T},m)$ with $i < j$ being the first two cars preferring $z$. Define $\breve{P}_{0\&1}(T,m)^{i,j}$ to be the set of elements of $\breve{P}_{0\&1}(T,m)$ with $i < j$ being the first two cars preferring a single vertex such that $j$ parks at $z$. Notice that $$\breve{P}_{\geq2}(\tilde{T},m) = \bigsqcup_{i<j}{\breve{P}_{\geq2}(\tilde{T},m)^{i,j}}$$ and $$\breve{P}_{0\&1}(T,m) \supset \bigsqcup_{i<j}{\breve{P}_{0\&1}(T,m)^{i,j}};$$ the latter set inequality is strict as $\breve{P}_{0\&1}(T,m)$ includes the situation of more than two cars preferring a single vertex such that the last car parks at $z$.

In both $\breve{P}_{\geq2}(\tilde{T},m)^{i,j}$ and $\breve{P}_{0\&1}(T,m)^{i,j}$, the preferences of the cars in $U_{i,j}$ must constitute a parking function on the respective trees. For each such restricted parking function $f_{\tilde{T}}: U_{i,j} \rightarrow [n]-\{z\}$ on $\tilde{T}$, let $\breve{P}_{\geq2}(\tilde{T},m)^{i,j}[f_{\tilde{T}}]$ denote the set of elements of $\breve{P}_{\geq2}(\tilde{T},m)^{i,j}$ whose $U_{i,j}$ preferences are determined by $f_{\tilde{T}}$. For each such restricted parking function $f_{T}: U_{i,j} \rightarrow [n]-\{z\}$ on $T$, let $\breve{P}_{0\&1}(T,m)^{i,j}[f_{T}]$ denote the set of elements of $\breve{P}_{0\&1}(T,m)^{i,j}$ whose $U_{i,j}$ preferences are determined by $f_{T}$.

Applying $\mathrm{flip}^*$ to any restricted parking function $f_{\tilde{T}}: U_{i,j} \rightarrow [n]-\{z\}$ on $\tilde{T}$ yields a restricted parking function $\mathrm{flip}^*(f_{\tilde{T}}): U_{i,j} \rightarrow [n]-\{z\}$ on $T$; the converse is also true. By the arguments for Case 2, $\mathrm{flip}^*$ yields a one-to-one correspondence between the sets \[\{\breve{P}_{\geq2}(\tilde{T},m)^{i,j}[f_{\tilde{T}}] \mid f_{\tilde{T}}: U_{i,j} \rightarrow [n]-\{z\} \; \mbox{parking function}\}\] and \[\{\breve{P}_{0\&1}(T,m)^{i,j}[f_{T}] \mid f_T: U_{i,j} \rightarrow [n]-\{z\} \; \mbox{parking function}\}.\]

Fix a restricted parking function $f_{\tilde{T}}: U_{i,j} \rightarrow [n]-\{z\}$ on $\tilde{T}$. We compare the corresponding sets $\breve{P}_{\geq2}(\tilde{T},m)^{i,j}[f_{\tilde{T}}]$ and $\breve{P}_{0\&1}(T,m)^{i,j}[\mathrm{flip}^*(f_{\tilde{T}})]$. This involves comparing the number of ways to complete $f_{\tilde{T}}$ to an element of $\breve{P}_{\geq2}(\tilde{T},m)^{i,j}$ and the number of ways to complete $\mathrm{flip}^*(f_{\tilde{T}})$ to an element of $\breve{P}_{0\&1}(T,m)^{i,j}$. Since $m \leq \min(|N(z)|,  \mathrm{minleafdist}(T)  )$, we immediately have $m^2 \leq n$.

Denote by $\overline{f_{\tilde{T}}}$ and $\overline{\mathrm{flip}^*(f_{\tilde{T}})}$ the parking functions resulting from including the preferences of the car pair $(i,j)$ in $f_{\tilde{T}}$ and $\mathrm{flip}^*(f_{\tilde{T}})$, respectively. Note that the cars $(i,j)$ must prefer $z$ in $\overline{f_{\tilde{T}}}$, whereas they can prefer at least $|N(z)|$ vertices (which allow car $j$ to park at $z$) in $\overline{\mathrm{flip}^*(f_{\tilde{T}})}$, since $m \leq \min(|N(z)|,  \mathrm{minleafdist}(T)  )$. The idea here is that $\overline{f_{\tilde{T}}}$ and $\overline{\mathrm{flip}^*(f_{\tilde{T}})}$ are ``almost the same'' except for the preference of the cars $(i,j)$, so we can focus on the differences between completing $\overline{f_{\tilde{T}}}$ and $\overline{\mathrm{flip}^*(f_{\tilde{T}})}$ to an element of $\breve{P}_{\geq2}(\tilde{T},m)^{i,j}$ and an element of $\breve{P}_{0\&1}(T,m)^{i,j}$, respectively.

Let $m' \geq |N(z)|$ be the number of possible vertices for $(i,j)$ to prefer so that $j$ parks at $z$, in the formation of $\overline{\mathrm{flip}^*(f_{\tilde{T}})}$. This yields the partition of $\breve{P}_{0\&1}(T,m)^{i,j}[\mathrm{flip}^*(f_{\tilde{T}})]$ into the sets $\breve{P}_{0\&1}(T,m)^{i,j}_0[\mathrm{flip}^*(f_{\tilde{T}})]$, $\breve{P}_{0\&1}(T,m)^{i,j}_1[\mathrm{flip}^*(f_{\tilde{T}})]$, $\ldots$, $\breve{P}_{0\&1}(T,m)^{i,j}_{m'-1}[\mathrm{flip}^*(f_{\tilde{T}})]$ where $(i,j)$ prefers vertex $v_k$ in $\breve{P}_{0\&1}(T,m)^{i,j}_{k}[\mathrm{flip}^*(f_{\tilde{T}})]$. Notice that if $(i,j)$ prefers $v_k$, then no car in $[m]-[j]$ can prefer any vertex on $\mathrm{path}(v_k,z)$, since $\mathrm{path}(v_k,z)$ will have been completely occupied by vertices in $[j]$. See Example \ref{illusthm} for illustration. $\overline{\mathrm{flip}^*(f_{\tilde{T}})}$ is then completed to an element of $\breve{P}_{0\&1}(T,m)^{i,j}$ by determining the preferences of the remaining cars $[m]-[j]$.

On the other hand, $\overline{f_{\tilde{T}}}$ is completed to an element of $\breve{P}_{\geq2}(\tilde{T},m)^{i,j}$ by choosing $0 \leq l \leq m-j$ cars to prefer non-root vertices, while the remaining $m-j-l$ cars will prefer $z$; the preferences for $[j]$ have already been determined with $(i,j)$ preferring $z$. This yields the partition of $\breve{P}_{\geq2}(\tilde{T},m)^{i,j}[f_{\tilde{T}}]$ into the sets $\breve{P}_{\geq2}(\tilde{T},m)^{i,j}_0[f_{\tilde{T}}]$, $\breve{P}_{\geq2}(\tilde{T},m)^{i,j}_1[f_{\tilde{T}}]$, $\ldots$, $\breve{P}_{\geq2}(\tilde{T},m)^{i,j}_{m-j}[f_{\tilde{T}}]$ depending on $l$.

We now show that $P_{\geq2}(\tilde{T},m)^{i,j}[f_{\tilde{T}}] < P_{0\&1}(T,m)^{i,j}[\mathrm{flip}^*(f_{\tilde{T}})]$, by comparing $\breve{P}_{\geq2}(\tilde{T},m)^{i,j}_l[f_{\tilde{T}}]$ and $\breve{P}_{0\&1}(T,m)^{i,j}_l[\mathrm{flip}^*(f_{\tilde{T}})]$ for $l = 0, 1, \ldots, m-j$. Without loss of generality assume that the $\breve{P}_{0\&1}(T,m)^{i,j}_l[\mathrm{flip}^*(f_{\tilde{T}})]$ are ordered by decreasing cardinality.

Let $\breve{P}_{0\&1}(T,j+l)^{i,j}_l[\mathrm{flip}^*(f_{\tilde{T}})]$ denote the subset of $\breve{P}(T,j+l)$ whose cars $[j]$ have already made their preferences according to $\overline{\mathrm{flip}^*(f_{\tilde{T}})}$ and the cars $(i,j)$ prefer $v_l$, for $0 \leq l \leq m-j$. By Lemma \ref{crudebounds}(\ref{pfestimates}), we have the bound \[P_{0\&1}(T,j+l)^{i,j}_l[\mathrm{flip}^*(f_{\tilde{T}})] \leq \frac{P_{0\&1}(T,m)^{i,j}_{l}[\mathrm{flip}^*(f_{\tilde{T}})]}{(n-j-l)^{\overline{m-j-l}}}.\]

$\overline{f_{\tilde{T}}}$ is completed to an element of $\breve{P}_{\geq2}(\tilde{T},m)^{i,j}_{l}[f_{\tilde{T}}]$ by choosing $l$ cars from $[m]-[j]$ and then determining their preferences on $\tilde{T}-\{z\}$; the other $m-j-l$ cars will prefer $z$. Since $m \leq   \mathrm{minleafdist}(T)  $, there is only one outgoing path from $\mathrm{flip}^*(v_l)$ in $\tilde{T}$, which leads to a leaf. For any $s \in \breve{P}_{\geq2}(\tilde{T},m)^{i,j}_{l}[f_{\tilde{T}}]$, the outgoing path from $\mathrm{flip}^*(v_l)$ will be completely occupied by cars in $[j]$ with the exception of $\mathrm{flip}^*(v_l)$ itself. After choosing the $l$ cars, let $\breve{P}(\tilde{T}-\{z\},l)[f_{\tilde{T}}]$ denote the subset of $\breve{P}_{\geq2}(\tilde{T},m)^{i,j}_{l}[f_{\tilde{T}}]$ where these specific $l$ cars are the ones chosen; we use this notation for simplicity since the number of ways to determine the preferences of the $l$ cars is independent of the specific cars chosen, and is always equal to $|\breve{P}(\tilde{T}-\{z\},l)[f_{\tilde{T}}]| = P(\tilde{T}-\{z\},l)[f_{\tilde{T}}]$. We can partition $\breve{P}(\tilde{T}-\{z\},l)[f_{\tilde{T}}]$ as $$\breve{P}(\tilde{T}-\{z\},l)[f_{\tilde{T}}] = \breve{P}(\tilde{T}-\{z\},l)[f_{\tilde{T}},\widehat{\mathrm{flip}^*(v_l)}] \sqcup \breve{P}(\tilde{T}-\{z\},l)[f_{\tilde{T}},\overline{\mathrm{flip}^*(v_l)}],$$ where $\mathrm{flip}^*(v_l)$ is never occupied in the former set while $\mathrm{flip}^*(v_l)$ may be occupied in the latter set. Notice that $\mathrm{flip}^*$ is an injection from $\breve{P}(\tilde{T}-\{z\},l)[f_{\tilde{T}},\widehat{\mathrm{flip}^*(v_l)}]$ to $\breve{P}_{0\&1}(T,j+l)^{i,j}_l[\mathrm{flip}^*(f_{\tilde{T}})]$ where the preference of $(i,j)$ is changed from $z$ to $v_l$, hence $P(\tilde{T}-\{z\},l)[f_{\tilde{T}},\widehat{\mathrm{flip}^*(v_l)}] \leq P_{0\&1}(T,j+l)^{i,j}_l[\mathrm{flip}^*(f_{\tilde{T}})]$. Since $m \leq \mathrm{minleafdist}(T)$, there are at least $m-j$ vertices connecting $\mathrm{flip}^*(v_l)$ to any vertex of degree at least three. By Lemma \ref{crudebounds}(\ref{mustpark}) we have $P(\tilde{T}-\{z\},l)[f_{\tilde{T}},\overline{\mathrm{flip}^*(v_l)}] \leq \frac{lm}{n-j-l+2}P(\tilde{T}-\{z\},l)[f_{\tilde{T}},\widehat{\mathrm{flip}^*(v_l)}] \leq \frac{lm}{n-j-l+2}P_{0\&1}(T,j+l)^{i,j}_l[\mathrm{flip}^*(f_{\tilde{T}})] < P_{0\&1}(T,j+l)^{i,j}_l[\mathrm{flip}^*(f_{\tilde{T}})]$, since $lm \leq (m-j)m = m^2-mj < m^2-(j+l)+1 \leq n-j-l+1$. It follows that \begin{align*}
P_{\geq2}(\tilde{T},m)^{i,j}_{l}[f_{\tilde{T}}] &= {m-j \choose l}P(\tilde{T}-\{z\},l)[f_{\tilde{T}}] \\ &= {m-j \choose l}[P(\tilde{T}-\{z\},l)[f_{\tilde{T}},\widehat{\mathrm{flip}^*(v_l)}]+P(\tilde{T}-\{z\},l)[f_{\tilde{T}},\overline{\mathrm{flip}^*(v_l)}]] \\ &< 2{m-j \choose l}P_{0\&1}(T,j+l)^{i,j}_l[\mathrm{flip}^*(f_{\tilde{T}})] \\ &= 2{m-j \choose m-j-l}P_{0\&1}(T,j+l)^{i,j}_l[\mathrm{flip}^*(f_{\tilde{T}})] \\ &= \frac{2(m-j)(m-j-1)\cdots(l+1)}{(m-j-l)(m-j-l-1)\cdots1}P_{0\&1}(T,j+l)^{i,j}_l[\mathrm{flip}^*(f_{\tilde{T}})]. \end{align*}

For $l < m-j$, the inequality established above yields \begin{align*}
P_{\geq2}(\tilde{T},m)^{i,j}_{l}[f_{\tilde{T}}] &< \frac{2(m-j)(m-j-1)\cdots(l+1)}{(m-j-l)(m-j-l-1)\cdots1}P_{0\&1}(T,j+l)^{i,j}_l[\mathrm{flip}^*(f_{\tilde{T}})] \\ &< (n-j-l)(n-j-l-1)\cdots(n-m+1)P_{0\&1}(T,j+l)^{i,j}_l[\mathrm{flip}^*(f_{\tilde{T}})] \\ &\leq P_{0\&1}(T,m)^{i,j}_{l}[\mathrm{flip}^*(f_{\tilde{T}})], \end{align*} where the second inequality follows by comparing the $m-j-l$ corresponding factors on both sides, because $n-m+1 \geq m^2-m+1 > 2m-4 \geq 2(m-j)$.

It remains to consider the case $l = m-j$. By the inequalities established above, we have $P_{\geq2}(\tilde{T},m)^{i,j}_{m-j}[f_{\tilde{T}}] < 2{m-j \choose m-j}P_{0\&1}(T,m)^{i,j}_{m-j}[\mathrm{flip}^*(f_{\tilde{T}})] = 2P_{0\&1}(T,m)^{i,j}_{m-j}[\mathrm{flip}^*(f_{\tilde{T}})]$. To establish an inequality between $P_{\geq2}(\tilde{T},m)^{i,j}[f_{\tilde{T}}]$ and $P_{0\&1}(T,m)^{i,j}[\mathrm{flip}^*(f_{\tilde{T}})]$, it suffices to compare $P_{\geq2}(\tilde{T},m)^{i,j}_{m-j}[f_{\tilde{T}}]+P_{\geq2}(\tilde{T},m)^{i,j}_{0}[f_{\tilde{T}}]$ with $P_{0\&1}(T,m)^{i,j}_{m-j}[\mathrm{flip}^*(f_{\tilde{T}})]+P_{0\&1}(T,m)^{i,j}_0[\mathrm{flip}^*(f_{\tilde{T}})]$. Since $P_{0\&1}(T,m)^{i,j}_{m-j}[\mathrm{flip}^*(f_{\tilde{T}})] \leq P_{0\&1}(T,m)^{i,j}_{0}[\mathrm{flip}^*(f_{\tilde{T}})]$ by assumption of decreasing order, we have \begin{align*}
P_{\geq2}(\tilde{T},m)^{i,j}_{m-j}[f_{\tilde{T}}]+P_{\geq2}(\tilde{T},m)^{i,j}_{0}[f_{\tilde{T}}] &= P_{\geq2}(\tilde{T},m)^{i,j}_{m-j}[f_{\tilde{T}}]+1 \\ &\leq 2P_{0\&1}(T,m)^{i,j}_{m-j}[\mathrm{flip}^*(f_{\tilde{T}})]-1+1 \\ &= 2P_{0\&1}(T,m)^{i,j}_{m-j}[\mathrm{flip}^*(f_{\tilde{T}})] \\ &\leq P_{0\&1}(T,m)^{i,j}_{m-j}[\mathrm{flip}^*(f_{\tilde{T}})]+P_{0\&1}(T,m)^{i,j}_{0}[\mathrm{flip}^*(f_{\tilde{T}})].\end{align*}

Since $P_{\geq2}(\tilde{T},m)^{i,j}_{m-j}[f_{\tilde{T}}]+P_{\geq2}(\tilde{T},m)^{i,j}_{0}[f_{\tilde{T}}] \leq P_{0\&1}(T,m)^{i,j}_{m-j}[\mathrm{flip}^*(f_{\tilde{T}})]+P_{0\&1}(T,m)^{i,j}_{0}[\mathrm{flip}^*(f_{\tilde{T}})]$ and since $P_{\geq2}(\tilde{T},m)^{i,j}_{l}[f_{\tilde{T}}] < P_{0\&1}(T,m)^{i,j}_{l}[\mathrm{flip}^*(f_{\tilde{T}})]$ for $l < m-j$, it follows that $P_{\geq2}(\tilde{T},m)^{i,j}[f_{\tilde{T}}] < P_{0\&1}(T,m)^{i,j}[\mathrm{flip}^*(f_{\tilde{T}})]$ whenever $m-j \geq 2$ (also trivially true when $m-j = 0$). Summing over all $i < j$ and all $f_{\tilde{T}}$, we obtain \[\sum_{i<j}{\sum_{f_{\tilde{T}}}{P_{\geq2}(\tilde{T},m)^{i,j}[f_{\tilde{T}}]}} < \sum_{i<j}{\sum_{f_{\tilde{T}}}{P_{0\&1}(T,m)^{i,j}[\mathrm{flip}^*(f_{\tilde{T}})]}}\] and hence $P_{\geq2}(\tilde{T},m) < P_{0\&1}(T,m)$.

In conclusion, summing over the three cases yields $P(\tilde{T},m) < P(T,m)$, with the difference coming from Case 3.

\end{proof}

\begin{rem}
In the case that $T$ is a starlike tree (a collection of paths joined at a common root), the same conclusion holds with the looser restriction $m \leq \min(N(z),\sqrt{n})$; the same argument applies with slight modifications.
\end{rem}

\begin{exm}
\label{illusthm}
Suppose $\tilde{T}$ is the following digraph.

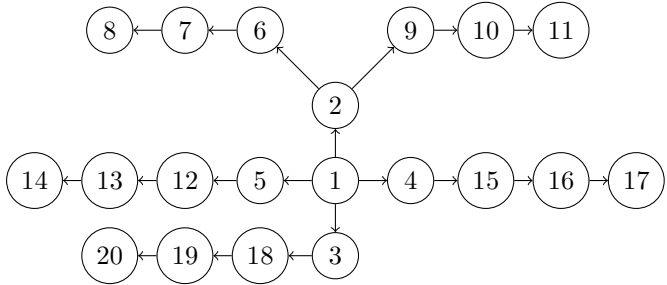
\begin{tikzpicture}
\node[shape=circle,draw=black] (1) at (0,0) {1};
\node[shape=circle,draw=black] (2) at (0,1) {2};
\node[shape=circle,draw=black] (3) at (0,-1) {3};
\node[shape=circle,draw=black] (4) at (1,0) {4};
\node[shape=circle,draw=black] (5) at (-1,0) {5};
\node[shape=circle,draw=black] (6) at (-1,2) {6};
\node[shape=circle,draw=black] (7) at (-2,2) {7};
\node[shape=circle,draw=black] (8) at (-3,2) {8};
\node[shape=circle,draw=black] (9) at (1,2) {9};
\node[shape=circle,draw=black] (10) at (2,2) {10};
\node[shape=circle,draw=black] (11) at (3,2) {11};
\node[shape=circle,draw=black] (12) at (-2,0) {12};
\node[shape=circle,draw=black] (13) at (-3,0) {13};
\node[shape=circle,draw=black] (14) at (-4,0) {14};
\node[shape=circle,draw=black] (15) at (2,0) {15};
\node[shape=circle,draw=black] (16) at (3,0) {16};
\node[shape=circle,draw=black] (17) at (4,0) {17};
\node[shape=circle,draw=black] (18) at (-1,-1) {18};
\node[shape=circle,draw=black] (19) at (-2,-1) {19};
\node[shape=circle,draw=black] (20) at (-3,-1) {20};

\path [->] (1) edge (2);
\path [->] (1) edge (3);
\path [->] (1) edge (4);
\path [->] (1) edge (5);
\path [->] (2) edge (6);
\path [->] (6) edge (7);
\path [->] (7) edge (8);
\path [->] (2) edge (9);
\path [->] (9) edge (10);
\path [->] (10) edge (11);
\path [->] (5) edge (12);
\path [->] (12) edge (13);
\path [->] (13) edge (14);
\path [->] (4) edge (15);
\path [->] (15) edge (16);
\path [->] (16) edge (17);
\path [->] (3) edge (18);
\path [->] (18) edge (19);
\path [->] (19) edge (20);
\end{tikzpicture}

We have $\mathrm{minleafdist}(T) = 4$, so $P(\tilde{T},m) < P(T,m)$ for $m = 2$, $3$, $4$ by Theorem \ref{sparsepark}. For $m = 4$, $(i,j) = (1,3)$, and $f_{\tilde{T}} = (8)$, we have $\mathrm{flip}^*(f_{\tilde{T}}) = (2)$, and the vertices for $(i,j)$ to prefer in $T$ are $3, 4, 5, 6, 9$.
\end{exm}

\section{Future Work}
For a source star $\tilde{S}$, computations show that we may have $P(\tilde{S},m) > P(S,m)$ for values of $m$ close to $n$. It would be interesting to find out how large $m$ needs to be for $P(\tilde{S},m) > P(S,m)$.

For a source tree $\tilde{T}$ with $m > \mathrm{minleafdist}(\tilde{T})$, the argument in Theorem \ref{sparsepark} breaks down since the $\mathrm{flip}^*$ operation may not yield a parking function on $T$. It would be interesting to see how/if this difficulty could be overcome, and what inequality could be established for larger values of $m$. Conceivably, the restriction $m \leq |N(z)|$ in Theorem \ref{sparsepark} may be relaxed if finer estimates than those in Lemma \ref{crudebounds} are found or if the full size of $\breve{P}_{0\&1}(T,m)$ is considered by taking into account the case where more than two cars prefer a single vertex with the last car parking at $z$.


\begin{thebibliography}{10}
\bibitem{ArmLeoWar}
D. Armstrong, N. A. Loehr, and G. S. Warrington: \emph{Rational parking functions and catalan numbers}. Annals of Combinatorics, \textbf{20}(1), 21--58 (2016)
\bibitem{Baum}
Alyson Baumgardner: \emph{The naples parking function}. Honors Contract-Graph Theory, Florida Gulf Coast University (2019)
\bibitem{CaJoSch}
T. P. Peter J Cameron, Daniel Johannsen and P. Schweitzer: \emph{Counting defective parking functions}. Electronic Journal of Combinatorics, \textbf{15} (2008)
\bibitem{GesSeo}
I. M. Gessel and S. Seo: \emph{A refinement of Cayley's formula for trees}. Electronic Journal of Combinatorics, Vol. 11, No. 2 (2006)
\bibitem{KingYan}
W. King and C. H. Yan: \emph{Parking functions on directed graphs and some directed trees}. Preprint, {\tt arXiv:1905.12010 [math.CO]}
\bibitem{KonWei}
A. G. Konheim and B. Weiss: \emph{An occupancy discipline and applications}. SIAM Journal on Applied Mathematics, \textbf{14}(6), 1266--1274 (1966)
\bibitem{LackPan}
M.-L Lackner and A. Panholzer. \emph{Parking functions for mappings}. Journal of Combinatorial Theory, Series A, \textbf{142}, 1--28 (2016)
\bibitem{Stan}
R. Stanley: \emph{Enumerative Combinatorics}, Vol. 2. Cambridge University Press (1999)
\bibitem{Yan15}
C. H. Yan: \emph{Parking functions}, In M. B\`ona, editor. Handbook of Enumerative Combinatorics, chapter 13, pages 835--894. CRC Press (2015)
\end{thebibliography}
\end{document}